\documentclass[11pt,a4paper,twoside,reqno]{amsart}

\usepackage[T1]{fontenc}
\usepackage[utf8]{inputenc}
\usepackage{lmodern}
\usepackage{microtype}
\usepackage{amsmath,amssymb,amsthm,mathrsfs}
\usepackage{cite}
\usepackage{float}
\usepackage{enumitem}
\usepackage{tikz}
\usetikzlibrary{arrows.meta}
\usepackage[colorlinks=true,linkcolor=blue,citecolor=blue,urlcolor=blue,
  hypertexnames=false]{hyperref}
\usepackage{iftex}

\makeatletter

\@namedef{subjclassname@2020}{%
  \textup{2020} Mathematics Subject Classification}
\def\@settitle{\begin{center}%
  \baselineskip14\p@\relax
  \bfseries\@title
  \end{center}%
}
\makeatother

\ifPDFTeX
  \input{glyphtounicode}
\fi

\allowdisplaybreaks

\numberwithin{equation}{section}

\theoremstyle{plain}
\newtheorem{theorem}{Theorem}[section]
\newtheorem{lemma}[theorem]{Lemma}
\newtheorem{corollary}[theorem]{Corollary}
\theoremstyle{remark}
\newtheorem*{remark}{Remark}

\newcommand{\R}{\mathbb{R}}
\newcommand{\B}{\mathbb{B}}
\newcommand{\Sph}{\mathbb{S}}
\newcommand{\cC}{\mathcal{C}}
\newcommand{\cH}{\mathcal{H}}
\newcommand{\tr}{\operatorname{tr}}
\newcommand{\diver}{\operatorname{div}}

\hypersetup{
  pdftitle={Optimal transport and the ABP method in higher codimension},
  pdfauthor={Bang-Xian Han and Zhe-Feng Xu},
  pdfsubject={Optimal transport and the ABP method for submanifolds in higher codimension},
  pdfkeywords={optimal transport, Alexandrov--Bakelman--Pucci method,
  normal fibers, Michael--Simon--Sobolev inequality, Wulff inequality,
    ellipsoid, sharp constant, spherical curvature measure}
}

\title[Optimal transport and the ABP method]
{Optimal Transport and the ABP Method in Higher Codimension}

\author{Bang-Xian Han\textsuperscript{1}}
\address{\textsuperscript{1} School of Mathematics,
Shandong University,
Jinan 250100, China}
\email{hanbx@sdu.edu.cn}

\author{Zhe-Feng Xu\textsuperscript{2,3}}
\address{\textsuperscript{2} School of Mathematical Sciences,
University of Science and Technology of China,
Hefei 230026, China}
\address{\textsuperscript{3} SISSA,
Via Bonomea 265,
34136 Trieste, Italy}
\email{xzf1998@mail.ustc.edu.cn}
\email{zxu@sissa.it}


\date{\today}

\begin{document}

\begin{abstract}
We prove an inverse optimal-transport principle for the ABP contact set of an
immersed submanifold.  A fixed function on the submanifold determines, for
every probability density on a bounded convex target, a source measure and an
optimal plan.  For a uniform ellipsoidal target, disintegration over the
submanifold gives
conditional measures whose fiberwise $L^\infty$-norms satisfy a sharp
weighted average estimate with constant
$\omega_n/\omega_{n+m}$ in codimensions $1\leq m\leq4$.  This yields the
sharp ellipsoidal Michael--Simon--Sobolev inequality in the same range,
including its equality cases and symmetrization consequences.  To the best of
our knowledge, even in the Euclidean case, the sharp constants in
codimensions three and four were previously unknown.
\end{abstract}
\maketitle


\tableofcontents

\section{Introduction}

\subsection{Motivation and background}

The Alexandrov--Bakelman--Pucci method covers a prescribed target by the ABP
map and estimates its Jacobian
\cite{CaffarelliCabre,Cabre,CabreRosOtonSerra}.  It has proved particularly
effective for sharp geometric inequalities \cite{CabreRosOtonSerra,BrendleJAMS}.
Optimal transport provides a parallel approach to sharp functional and
isoperimetric inequalities
\cite{McCannDisplacement,OttoVillani,CorderoErausquinNazaretVillani,
FigalliMaggiPratelli}.  For an immersed submanifold the ABP map is defined on
the normal bundle and therefore changes dimension.  The same geometry appears
in transport proofs of Michael--Simon type inequalities
\cite{Castillon,BrendleEichmair,WangOT} and in transport between unequal
dimensions \cite{McCannPass}.

We follow the inverse viewpoint of \cite{HanABP}: the function defining the
contact set is prescribed, and the marginals are recovered from it.  Theorem
\ref{Inverse transport theorem} constructs, for every target density, an
optimal plan supported on the corresponding set of contact pairs.
Its reverse description by a convex potential is the standard
Brenier--McCann representation \cite{Brenier,McCann}.
Disintegration over the submanifold gives probability measures on the fibers
of the contact set in the normal spaces.  We estimate the $L^\infty$ norms of
their densities directly.  This transport construction works in every
codimension and for every bounded convex target; the restrictions on the
target and the codimension enter only in the fiber estimate.

For a uniform ellipsoidal target, a linear change of variables sends each
normal slice to a Euclidean ball.  The main case is codimension four, where a
sharp radial estimate follows from the spherical Steiner formula and
Gauss--Bonnet; lower codimensions follow by adding flat normal directions.
The special role of dimension four is explained after Lemma
\ref{Spherical Lemma}.

As an application of the transport estimate, we obtain the sharp ellipsoidal
Michael--Simon--Sobolev inequality in codimensions at most four.  The classical
inequality was proved independently by Michael--Simon \cite{MichaelSimon} and
Allard \cite{Allard}; see also the extension of Hoffman--Spruck
\cite{HoffmanSpruck}.  Brendle obtained the sharp Euclidean constant in
codimension two \cite{BrendleJAMS}, and the codimension-one case follows by
adding a flat normal direction.  To the best of our knowledge, the sharp
constant had not previously been determined in any codimension $m\geq3$; see
\cite{BaloghKristalyMester,Sun,BrendleSurvey}.  Corollary \ref{Main Corollary}
gives the Euclidean cases $m=3,4$ and the ellipsoidal inequality for all
$m\leq4$.  In the minimal case it extends the sharp ellipsoidal Wulff
inequality of Du--Yi--Zhao \cite{DuYiZhao} from codimensions one and two to
codimensions at most four.  The Wulff inequality and its ABP and transport
proofs provide the corresponding flat anisotropic background
\cite{FonsecaMuller,CabreRosOtonSerra,FigalliMaggiPratelli}.
Related sharp inequalities for hypersurfaces and isoperimetric problems in
curved ambient spaces have also been studied by potential-theoretic and
variational methods
\cite{AgostinianiFogagnoloMazzieriSharp,
AgostinianiFogagnoloMazzieriMinkowski,FogagnoloMazzieriPinamonti,
AntonelliFogagnoloPozzetta}.

\subsection{Main results}

We write $\mathrm{II}$ for the second fundamental form and
$H=\operatorname{tr}_\Sigma\mathrm{II}$ for the unnormalized mean curvature
vector.  The induced measures on $\Sigma$ and $\partial\Sigma$ are denoted by
$\mathrm d\mu$ and $\mathrm d\sigma$, and $\eta$ is the outward unit conormal.
For a bounded convex set $L$, let
\[
 h_L(z)=\sup_{\xi\in L}\langle z,\xi\rangle.
\]
We use the standard conventions for support functions, projections, and
Hausdorff convergence from \cite{Schneider}.
We write $\B_R^k$ and $\Sph^{k-1}$ for the open Euclidean ball of radius $R$
and the unit sphere in $\R^k$, respectively, and abbreviate $\B_1^k$ to $\B^k$ and
$\omega_k=|\B^k|$.  The notation $|\cdot|$ denotes Euclidean norm or induced
volume as dictated by context.  All ellipsoids below are open; taking their
closures does not change their volumes or support functions.  We use
$\operatorname{proj}_V$ for orthogonal projection onto a linear subspace $V$
and $\mathbf1_A$ for the indicator of a set $A$.  Lebesgue measure on the
ambient space and on each normal space $N_x\Sigma$ is denoted by $\mathrm
d\xi$ and $\mathrm dy$, respectively.

Let $X:\Sigma^n\to\R^{n+m}$ be a smooth immersion of a compact manifold,
possibly with smooth boundary, let $\mathcal K\subset\R^{n+m}$ be a bounded
open convex set, and let $u\in C^{2,\alpha}(\Sigma)$, understood up to the
boundary, for some $\alpha\in(0,1)$.  If $\partial\Sigma\ne\varnothing$, assume
\[
 \langle\nabla^\Sigma u,\eta\rangle=h_{\mathcal K}(\eta)
 \qquad\text{on }\partial\Sigma.
\]
This condition plays the role, in support-function form, of the target
condition in the second boundary value problem for Monge--Amp\`ere equations
and optimal transport \cite{CaffarelliBoundary,Urbas,TrudingerWang}.  In the
ABP argument it excludes
boundary minimizers of $u-\langle X,\xi\rangle$ for every
$\xi\in\mathcal K$, and hence ensures that the contact map covers
$\mathcal K$.  For $\mathcal K=\B^{n+m}$ it reduces to
$\partial_\eta u=1$, the Neumann condition in the classical ABP construction
\cite{Cabre,CabreRosOtonSerra,BrendleJAMS}.

Denote
\[
 v(x)=\nabla^\Sigma u(x),\qquad
 Q_x(y)=D^2_\Sigma u(x)-\langle\mathrm{II}(x),y\rangle,
 \qquad \Phi(x,y)=v(x)+y,
\]
and define the global contact set
\begin{equation}\label{global contact set}
\begin{split}
 \cC=\big\{(x,y):\;&x\in\operatorname{int}\Sigma,\ y\in N_x\Sigma,
 \ \Phi(x,y)\in\mathcal K,\\
 &u(z)-\langle X(z),\Phi(x,y)\rangle
 \geq u(x)-\langle X(x),\Phi(x,y)\rangle
 \quad\text{for all }z\in\Sigma\big\}.
\end{split}
\end{equation}
For fixed $x$, write $\cC_x=\{y:(x,y)\in\cC\}$ and
\[
 \vartheta(x)=\#\{z\in\Sigma:X(z)=X(x),\ u(z)=u(x)\}.
\]
The multiplicity $\vartheta$ records self-intersections of $(X,u)$; the
factor $1/\vartheta$ below removes the resulting overcounting.
Figure \ref{ABP schematic} summarizes the geometry of one normal fiber.

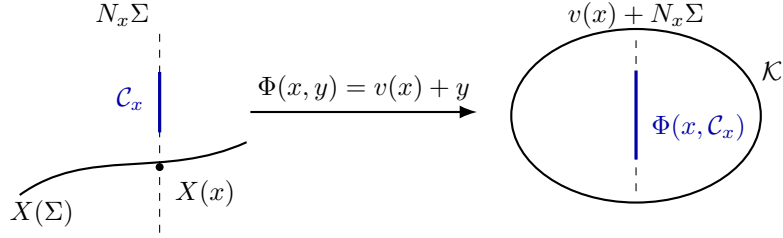
\begin{figure}[H]
\centering
\begin{tikzpicture}[x=1cm,y=1cm,>=Latex,font=\small]
 \draw[thick]
  (-5.3,-0.35) .. controls (-4.4,0.30) and (-3.4,-0.15) .. (-2.3,0.35);
 \node[below left] at (-4.45,-0.28) {$X(\Sigma)$};
 \coordinate (x) at (-3.45,0.02);
 \fill (x) circle (1.5pt);
 \node[below right] at (-3.38,-0.02) {$X(x)$};
 \draw[dashed] (-3.45,-0.85)--(-3.45,1.85);
 \node[above left] at (-3.45,1.76) {$N_x\Sigma$};
 \draw[very thick,blue!65!black] (-3.45,0.48)--(-3.45,1.28);
 \node[left,blue!65!black] at (-3.52,0.92) {$\cC_x$};

 \draw[->,thick] (-2.25,0.75)--(0.75,0.75)
  node[midway,above] {$\Phi(x,y)=v(x)+y$};

 \draw[thick] (2.85,0.70) ellipse (1.65 and 1.15);
 \node[right] at (4.38,1.30) {$\mathcal K$};
 \draw[dashed] (2.85,-0.30)--(2.85,1.70);
 \node[above] at (2.85,1.72) {$v(x)+N_x\Sigma$};
 \draw[very thick,blue!65!black] (2.85,0.12)--(2.85,1.30);
 \node[right,blue!65!black] at (2.92,0.55) {$\Phi(x,\cC_x)$};
\end{tikzpicture}
\caption{The ABP map sends the contact fiber over $x$ into the corresponding
affine normal slice of the target.}
\label{ABP schematic}
\end{figure}

The ABP map has the following transport property.

\begin{theorem}[Inverse transport for the ABP map]\label{Inverse transport theorem}
Let $m\geq1$, let $X:\Sigma^n\to\R^{n+m}$ be a smooth immersion of a
compact manifold, possibly with smooth boundary, and let
$\mathcal K\subset\R^{n+m}$ be a bounded open convex set.  Let
$u\in C^{2,\alpha}(\Sigma)$, understood up to the boundary, for some
$\alpha\in(0,1)$.  If $\partial\Sigma\ne\varnothing$, assume
$\langle\nabla^\Sigma u,\eta\rangle=h_{\mathcal K}(\eta)$ on
$\partial\Sigma$.  With $v$, $Q$, $\Phi$, $\cC$, and $\vartheta$ defined
above, the following conclusions hold.
\begin{enumerate}[label=\textup{(\roman*)}]
\item One has $Q_x(y)\geq0$ on $\cC$ and, for every Borel function
$\varphi:\mathcal K\to[0,\infty]$,
\begin{equation}\label{contact area formula}
 \int_{\cC}\varphi(\Phi(x,y))\frac{\det Q_x(y)}{\vartheta(x)}
 \,\mathrm d\mu(x)\,\mathrm dy
 =\int_{\mathcal K}\varphi(\xi)\,\mathrm d\xi.
\end{equation}
Both sides are understood as extended nonnegative integrals.

\item Let $\rho$ be a nonnegative Borel function satisfying
$\int_{\mathcal K}\rho\,\mathrm d\xi=1$, and set
\[
 I_\rho(x)=\int_{\cC_x}\rho(\Phi(x,y))\det Q_x(y)\,\mathrm dy,
 \qquad
 \mathrm d\lambda_{u,\rho}(x)=\frac{I_\rho(x)}{\vartheta(x)}\,\mathrm d\mu(x).
\]
By part \textup{(i)}, $I_\rho(x)<\infty$ for $\mu$-almost every $x$.
For $0<I_\rho(x)<\infty$, define
\begin{equation}\label{conditional measure}
 \mathrm d\nu_{x,\rho}(y)
 =\frac{\rho(\Phi(x,y))\det Q_x(y)}{I_\rho(x)}
 \mathbf1_{\cC_x}(y)\,\mathrm dy.
\end{equation}
On $\{I_\rho=0\}\cup\{I_\rho=\infty\}$ set
$\nu_{x,\rho}=\delta_0$.  Then $x\mapsto\nu_{x,\rho}$ is a Borel
probability kernel, $\lambda_{u,\rho}$ is a probability measure, and the plan
$\pi_{u,\rho}$ defined by
\begin{equation}\label{inverse transport plan}
 \int \zeta\,\mathrm d\pi_{u,\rho}
 =\int_\Sigma\int_{N_x\Sigma}
 \zeta(x,v(x)+y)\,\mathrm d\nu_{x,\rho}(y)\,\mathrm d\lambda_{u,\rho}(x)
\end{equation}
for bounded Borel $\zeta$ has marginals $\lambda_{u,\rho}$ and
$\rho\,\mathrm d\xi$, and is optimal for
$c(x,\xi)=\frac12|X(x)-\xi|^2$.

\item Let
\begin{equation}\label{Brenier potential}
 \Psi(\xi)=\max_{z\in\Sigma}
 \{\langle X(z),\xi\rangle-u(z)\}.
\end{equation}
As a supremum of affine functions, $\Psi$ is convex and hence differentiable
almost everywhere.  Then $T=\nabla\Psi$ is the Brenier map from
$\rho\,\mathrm d\xi$ to $X_\#\lambda_{u,\rho}$ and
\begin{equation}\label{reverse transport marginal}
 T_\#(\rho\,\mathrm d\xi)=X_\#\lambda_{u,\rho}.
\end{equation}
\end{enumerate}
\end{theorem}

\begin{remark}[Reduction to the classical ABP construction]
If $X:\overline\Omega\hookrightarrow\R^n$ is the inclusion of a bounded
smooth domain, corresponding formally to $m=0$, then the normal variable is
absent, $\mathrm{II}=0$, and $\vartheta=1$.  Thus
\[
 \Phi=\nabla u,\qquad Q=D^2u,
 \qquad
 \int_{\cC}\varphi(\nabla u)\det D^2u\,\mathrm dx
 =\int_{\mathcal K}\varphi(\xi)\,\mathrm d\xi,
\]
where $\cC$ is the usual lower contact set with slopes in $\mathcal K$.  This
is the classical ABP area formula.  More generally, if $m>0$ and $X(\Sigma)$
is contained in an affine $n$-plane parallel to $P$, then $\mathrm{II}=0$ and
the contact condition depends only on the projection of $\xi$ onto $P$.  The
classical formula for $\operatorname{proj}_P\mathcal K$ is recovered by
normalizing each normal slice by its $m$-dimensional volume and then
integrating out the normal variable.  By contrast, minimality alone does not
give this reduction: when $H=0$ but
$\mathrm{II}\ne0$, the matrix
$Q_x(y)=D^2_\Sigma u(x)-\langle\mathrm{II}(x),y\rangle$ still depends on the
normal variable.
\end{remark}

For the remainder of the paper we use normalized Lebesgue measure on the
target, $\rho=|\mathcal K|^{-1}\mathbf1_{\mathcal K}$.  Write
\[
 I(x)=\int_{\cC_x}\det Q_x(y)\,\mathrm dy,
 \qquad
 \mathrm d\lambda_u(x)=\frac{I(x)}{\vartheta(x)|\mathcal K|}\,\mathrm d\mu(x),
\]
and, when $I(x)>0$,
\[
 \mathrm d\nu_x(y)=\frac{\det Q_x(y)}{I(x)}
 \mathbf1_{\cC_x}(y)\,\mathrm dy.
\]
On $\{I=0\}\cup\{I=\infty\}$, set $\nu_x=\delta_0$ and $p_\infty(x)=0$.
The area formula gives $\mu(\{I=\infty\})=0$ and
$\lambda_u(\{I=0\}\cup\{I=\infty\})=0$, so this convention does not affect
the plan.  For $0<I(x)<\infty$, set
\[
 p_\infty(x)
 :=\left\|\frac{\mathrm d\nu_x}{\mathrm dy}\right\|_{L^\infty(N_x\Sigma)}
 =\frac1{I(x)}
 \big\|\det Q_x\,\mathbf1_{\cC_x}\big\|_{L^\infty(N_x\Sigma)}.
\]
Since $I(x)=0$ implies
$\|\det Q_x\mathbf1_{\cC_x}\|_{L^\infty}=0$,
\begin{equation}\label{conditional density identity}
 p_\infty(x)\,\mathrm d\lambda_u(x)
 =\frac{\|\det Q_x\mathbf1_{\cC_x}\|_{L^\infty(N_x\Sigma)}}
 {\vartheta(x)|\mathcal K|}\,\mathrm d\mu(x).
\end{equation}
For an ellipsoidal target, put each normal slice in Euclidean coordinates as follows.
Write the centered ellipsoid $\mathcal E$ as
\[
 \mathcal E=\{\xi\in\R^{n+m}:\langle G\xi,\xi\rangle<1\},
 \qquad G=G^{\mathsf T}>0.
\]
Relative to $\R^{n+m}=T_x\Sigma\oplus N_x\Sigma$, let
$G_{NN}(x)$ be the normal block.  Set
\[
 w_x=(\det G_{NN}(x))^{-1/2}.
\]
The factor $w_x$ is the Jacobian of the inverse change of variables on the
normal slice.  By \eqref{conditional density identity},
\[
 \int_\Sigma w_xp_\infty(x)\,\mathrm d\lambda_u(x)
 =\frac1{|\mathcal E|}\int_\Sigma
 \frac{w_x\|\det Q_x\mathbf1_{\cC_x}\|_{L^\infty(N_x\Sigma)}}
 {\vartheta(x)}\,\mathrm d\mu(x).
\]

\begin{theorem}[Sharp averaged fiberwise $L^\infty$ estimate]
\label{Fiberwise estimate theorem}
Assume the hypotheses of Theorem \ref{Inverse transport theorem} with
$1\leq m\leq4$ and $\mathcal K=\mathcal E$, where $\mathcal E$ is a centered
ellipsoid.  Then
\begin{equation}\label{fiberwise estimate}
 \int_\Sigma w_xp_\infty(x)\,\mathrm d\lambda_u(x)
 \geq\frac{\omega_n}{\omega_{n+m}}.
\end{equation}
The constant is sharp.  If $m=4$ and equality holds, then for $\mu$-almost
every $x$ with $p_\infty(x)>0$,
\begin{equation}\label{fiberwise rigidity}
 \widetilde{\cC}_x=\B^4_{R(x)},
 \qquad
 y\longmapsto\det Q_x(y)\text{ is constant on }\cC_x.
\end{equation}
Here $\widetilde{\cC}_x$ and $R(x)$ are defined in Section
\ref{section fiberwise estimate}.
\end{theorem}

The constant is attained for every fixed centered ellipsoid, not only for the
ball.  Indeed, let $P$ be any $n$-plane, put
$K_P=\operatorname{proj}_P\mathcal E$, take
$\Sigma=\overline{K_P}\subset P$ with the inclusion immersion, and set
$u(x)=|x|^2/2$.  Then $\vartheta\equiv1$, $Q_x\equiv I_P$, and $\cC_x$ is
the full normal slice $\{y\in P^\perp:x+y\in\mathcal E\}$.  In
addition, $\langle x,\eta\rangle=h_{K_P}(\eta)=h_{\mathcal E}(\eta)$ on
$\partial K_P$, so the boundary condition holds.  The factor $w_x$ is constant.
Writing
$c_{\mathcal E}=(\det G)^{-1/2}$, the Schur identity gives
$w_x|K_P|=c_{\mathcal E}\omega_n$, and hence
\[
 \int_\Sigma w_xp_\infty(x)\,\mathrm d\lambda_u(x)
 =\frac{w_x|K_P|}{|\mathcal E|}
 =\frac{c_{\mathcal E}\omega_n}{c_{\mathcal E}\omega_{n+m}}
 =\frac{\omega_n}{\omega_{n+m}}.
\]
This example proves sharpness.  For a centered ellipsoid
$\mathcal E\subset\R^{n+m}$, let
\[
 V_*(\mathcal E)=\min_{P\in\operatorname{Gr}(n,n+m)}
 |\operatorname{proj}_P\mathcal E|,
\]
where $\operatorname{Gr}(n,n+m)$ is the Grassmannian of $n$-planes and
$\operatorname{proj}_P$ denotes orthogonal projection.

\begin{corollary}[Sharp anisotropic Michael--Simon--Sobolev inequality]\label{Main Corollary}
Let $n\geq2$ and $1\leq m\leq4$.  Let
$X:\Sigma\to\R^{n+m}$ be a smooth immersion of a compact $n$-dimensional
manifold with smooth boundary, where the boundary may be empty, and let
$\mathcal E$ be a centered ellipsoid.  Then every positive
$f\in C^\infty(\Sigma)$ satisfies
\begin{equation}\label{Main inequality}
 \int_\Sigma h_{\mathcal E}(\nabla^\Sigma f+fH)\,\mathrm d\mu
 +\int_{\partial\Sigma}f h_{\mathcal E}(\eta)\,\mathrm d\sigma
 \geq nV_*(\mathcal E)^{\frac1n}
 \left(\int_\Sigma f^{\frac n{n-1}}\,\mathrm d\mu\right)^{\frac{n-1}{n}}.
\end{equation}
Equality holds if and only if $\Sigma$ is connected, $f$ is constant, and,
for some $n$-plane $P$ realizing $V_*(\mathcal E)$, $X$ is an embedding onto
$x_0+r\,\overline{\operatorname{proj}_P\mathcal E}$ for some
$x_0\in\R^{n+m}$ and $r>0$.
\end{corollary}

For $\mathcal E=\B^{n+m}$ this is the sharp Euclidean
Michael--Simon--Sobolev inequality.  Corollary \ref{Main Corollary} in
codimension four follows directly from Theorem
\ref{Fiberwise estimate theorem}; the cases $m<4$ follow by adjoining
flat normal directions.  For minimal immersions it gives the sharp
ellipsoidal Wulff inequality, and hence the anisotropic P\'olya--Szeg\H{o} and
$L^p$ consequences proved in Subsection~\ref{section symmetrization}.

The construction above makes sense in every codimension.  The ellipsoidal
assumption enters only in the fiber estimate, where the Schur complement
normalizes every normal slice to a Euclidean ball.  Extending this step to
general centrally symmetric targets is a natural question.  It also remains
open whether \eqref{fiberwise estimate} holds for $m>4$; the flat model shows
that no larger constant is possible.

\subsection{Organization}

Section~\ref{section inverse transport} proves the exact push-forward formula,
disintegrates it along the normal fibers, and identifies the resulting plan by
Kantorovich duality.  Section~\ref{section fiberwise estimate} first reduces
each ellipsoidal normal slice to a Euclidean ball.  The spherical Steiner
formula and Gauss--Bonnet then give the radial estimate in codimension four,
while differentiation of the truncated area formula converts it into the
averaged fiberwise bound.  Flat extension gives the lower codimensions.  In
Section~\ref{section applications}, a Neumann problem supplies the function
defining the contact set, and a trace identity together with the
arithmetic--geometric mean inequality yields the Michael--Simon--Sobolev
estimate.  The final part handles equality, lower-dimensional consequences,
and convex symmetrization.

\section{The inverse transport principle}\label{section inverse transport}

We prove Theorem \ref{Inverse transport theorem}, using the notation introduced
above.  The contact condition gives
\begin{equation}\label{QPositive}
 Q_x(y)\geq0\qquad\text{on }\cC.
\end{equation}
For fixed $x$, the global supporting inequalities defining $\cC_x$ are
closed affine half-spaces in $N_x\Sigma$.  Hence $\cC_x$ is their
intersection with the convex normal slice
$\{y\in N_x\Sigma:v(x)+y\in\mathcal K\}$, and is therefore convex and
relatively closed in that slice.  The minimum over
the compact set $\Sigma$ in \eqref{global contact set} depends continuously on
$(x,y)$, so $\cC$ is Borel.  Since $u\in C^{2,\alpha}$, the map $\Phi$ is
locally $C^{1,\alpha}$ on the normal bundle.

The map $F=(X,u):\Sigma\to\R^{n+m+1}$ is an immersion and is locally
injective.  Choose a finite open cover $U_1,\ldots,U_N$ of $\Sigma$ such that
$F|_{U_j}$ is injective, and refine it to a Borel partition
$A_1,\ldots,A_N$ with $A_j\subset U_j$.  Each fiber of $F$ meets every
$A_j$ in at most one point, so $\vartheta\leq N$.  By the Lusin--Souslin
theorem each $F(A_j)$ is Borel, and
\[
 \vartheta(x)=\sum_{j=1}^N\mathbf1_{F(A_j)}(F(x)).
\]
Thus $\vartheta$ is finite, uniformly bounded, and Borel.

\begin{proof}[Proof of Theorem \ref{Inverse transport theorem}]
Fix $\xi\in\mathcal K$ and consider
\[
 w_\xi(z)=u(z)-\langle X(z),\xi\rangle.
\]
At a boundary point,
\[
 \partial_\eta w_\xi
 =h_{\mathcal K}(\eta)-\langle\eta,\xi\rangle>0,
\]
since $\xi$ is an interior point of $\mathcal K$.  Thus no boundary point
minimizes $w_\xi$, and a minimum is attained at some
$x\in\operatorname{int}\Sigma$.  At such a point
$\operatorname{proj}_{T_x\Sigma}\xi=v(x)$, so
$y=\xi-v(x)\in N_x\Sigma$, and the second derivative test
gives $Q_x(y)\geq0$.  Hence $(x,y)\in\cC$ and $\Phi(x,y)=\xi$.

Choose geodesic orthonormal tangent fields and a normal frame parallel for the
normal connection at $x$.  In the associated normal-bundle coordinates, the
Weingarten formula gives
\[
 D\Phi(x,y)=
 \begin{pmatrix}
 Q_x(y)&0\\
 \mathrm{II}(v(x),\cdot)&I_{N_x\Sigma}
 \end{pmatrix}.
\]
Consequently, the Jacobian of $\Phi$ is
\begin{equation}\label{Phi jacobian}
 J\Phi(x,y)=\det Q_x(y)\qquad\text{on }\cC.
\end{equation}
Set
\[
 \psi_0(\xi)=\min_{z\in\Sigma}
 \bigl(u(z)-\langle X(z),\xi\rangle\bigr).
\]
The function $\psi_0$ is concave and Lipschitz on $\mathcal K$.  If $x$ is a
minimizer at $\xi$, then $-X(x)$ belongs to the superdifferential of
$\psi_0$ at $\xi$.  At every differentiability point all minimizers therefore
have the same image under $X$ and the same value of $u$, and every point in
this fiber of $(X,u)$ is a minimizer.  By Rademacher's theorem
\cite[\S3.1]{Federer}, for almost every $\xi\in\mathcal K$,
\begin{equation}\label{sum 1}
 \sum_{\substack{(x,y)\in\cC\\ \Phi(x,y)=\xi}}
 \frac1{\vartheta(x)}=1.
\end{equation}
The area formula \cite[\S3.2.3]{Federer}, together with
\eqref{Phi jacobian} and \eqref{sum 1}, proves
\eqref{contact area formula}.  Applying this identity with $\varphi=\rho$ shows
that $I_\rho<\infty$ for $\mu$-almost every $x$ and that
$\lambda_{u,\rho}(\Sigma)=1$.  In local trivializations of the normal
bundle, the density in \eqref{conditional measure} is jointly Borel;
with the choice $\delta_0$ on the exceptional set this defines a Borel
probability kernel.  Applying \eqref{contact area formula} with
$\varphi=\rho q$ gives the second marginal of \eqref{inverse transport plan} for
every bounded Borel function $q$ on $\mathcal K$.  The first marginal is
$\lambda_{u,\rho}$ by construction.

To prove optimality, define $\Psi$ by \eqref{Brenier potential} and
put
\[
 \phi(x)=\frac12|X(x)|^2-u(x),
 \qquad
 \psi(\xi)=\frac12|\xi|^2-\Psi(\xi).
\]
Then
\[
 \phi(x)+\psi(\xi)\leq\frac12|X(x)-\xi|^2.
\]
Equality holds if and only if $\xi=\Phi(x,y)$ for some $(x,y)\in\cC$.  The
plan $\pi_{u,\rho}$ is concentrated on these pairs, so Kantorovich duality
\cite{Brenier,McCann,Villani} gives its optimality.  At every differentiability
point $\xi$ of $\Psi$, all minimizers have the same image under $X$, and for
any such minimizer $x$, $\nabla\Psi(\xi)=X(x)$.
Combining this with \eqref{contact area formula}, weighted by $\rho$, gives
\eqref{reverse transport marginal}.
\end{proof}

\section{The four-dimensional fiber estimate}
\label{section fiberwise estimate}

In this section $m=4$, the target is a centered ellipsoid.  Write
\[
 \mathcal E=\{\xi\in\R^{n+4}:\langle G\xi,\xi\rangle<1\},
 \qquad G>0,
 \qquad c_{\mathcal E}=(\det G)^{-1/2}.
\]
Then $|\mathcal E|=c_{\mathcal E}\omega_{n+4}$.  Relative to the orthogonal
splitting $\R^{n+4}=T_x\Sigma\oplus N_x\Sigma$, write
\[
 G=\begin{pmatrix}G_{TT}&G_{TN}\\G_{NT}&G_{NN}\end{pmatrix},
 \qquad
 S_x=G_{TT}-G_{TN}G_{NN}^{-1}G_{NT}.
\]
The four blocks depend on $x$; we suppress this dependence in the notation.
The Schur complement $S_x$ is positive definite.  Set
\[
 w_x=(\det G_{NN})^{-1/2},
 \qquad
 V_x=|\operatorname{proj}_{T_x\Sigma}\mathcal E|,
\]
and
\[
 y_x^0=-G_{NN}^{-1}G_{NT}v(x),
 \qquad
 \tau_x=\langle S_xv(x),v(x)\rangle,
 \qquad
 R(x)=\sqrt{\max\{1-\tau_x,0\}}.
\]
Let $\Omega=\{x\in\operatorname{int}\Sigma:\tau_x<1\}$.  Completing the
square, or equivalently using the standard Schur-complement factorization
\cite[Chapter~1]{Bhatia}, gives
\begin{equation}\label{Schur identity}
 \langle G(v(x)+y),v(x)+y\rangle
 =\tau_x+\langle G_{NN}(y-y_x^0),y-y_x^0\rangle.
\end{equation}
Thus $\cC_x=\varnothing$ outside $\Omega$.  For $x\in\Omega$, set
\[
 \widetilde{\cC}_x
 =G_{NN}^{1/2}(\cC_x-y_x^0).
\]
This is a relatively closed convex subset of $\B^4_{R(x)}$.  The change of
variables $z=G_{NN}^{1/2}(y-y_x^0)$ gives $\mathrm dy=w_x\,\mathrm dz$.
If $I(x)>0$, the push-forward of $\nu_x$ under this affine change of variables
has a density with $L^\infty$ norm $w_xp_\infty(x)$.
Thus $\int_\Sigma w_xp_\infty\,\mathrm d\lambda_u$ is the average of the
$L^\infty$ norms of the conditional densities after this change of normal
coordinates.  Also,
\[
 \operatorname{proj}_{T_x\Sigma}\mathcal E
 =\{z\in T_x\Sigma:\langle S_xz,z\rangle<1\},
\]
so
\begin{equation}\label{Schur Jacobian}
 w_x=c_{\mathcal E}(\det S_x)^{1/2}
 =\frac{c_{\mathcal E}\omega_n}{V_x}.
\end{equation}

The following radial estimate is the four-dimensional part of the proof.  We
write $\mathsf d_{\Sph^3}$ for the geodesic distance on $\Sph^3$ and set
\[
 \mathsf d_{\Sph^3}(z,A)=\inf_{q\in A}\mathsf d_{\Sph^3}(z,q)
\]
for $A\subset\Sph^3$.  We write $\mathsf d_{\mathrm H}$ for the Hausdorff
distance induced by the metric of the ambient space.

\begin{lemma}\label{Spherical Lemma}
Let $U\subset\Sph^3$ be a smooth geodesically convex domain contained in an
open hemisphere.  Let $\kappa_1,\kappa_2\geq0$ be the principal curvatures of
$\partial U$ with respect to the outward unit conormal, let $\mathrm dA$
denote its area measure, and put $H_{\Sph^3}=\kappa_1+\kappa_2$.  Then
\begin{equation}\label{A}
 |U|+\frac12\int_{\partial U}H_{\Sph^3}\,\mathrm dA\leq\frac{|\Sph^3|}{2}.
\end{equation}
\end{lemma}

\begin{proof}
For $0<t<\pi/2$, let
$U_t=\{z\in\Sph^3:\mathsf d_{\Sph^3}(z,U)\leq t\}$.  If $p\in\partial U$ and $\nu_p$
is the outward unit conormal, geodesic convexity gives
$\langle\nu_p,q\rangle\leq0$ for every $q\in U$.
For $z=\cos s\,p+\sin s\,\nu_p$, where $0<s<\pi/2$, we have
\[
 \langle z,q\rangle
 =\cos s\langle p,q\rangle+\sin s\langle\nu_p,q\rangle
 \leq\cos s\qquad\text{for every }q\in U.
\]
Thus $p$ is a nearest point of $\overline U$ to $z$ and
$\mathsf d_{\Sph^3}(z,U)=s$.

The nearest point is unique at distance less than $\pi/2$.  Indeed, if
$p\ne q$ are nearest points to $z$ and
$\mathsf d_{\Sph^3}(z,p)=\mathsf d_{\Sph^3}(z,q)=\delta<\pi/2$, then the
short geodesic midpoint $m=(p+q)/|p+q|$ belongs to $\overline U$, while
\[
 \langle z,m\rangle
 =\frac{\cos\delta}{\cos(\mathsf d_{\Sph^3}(p,q)/2)}>\cos\delta,
\]
a contradiction.  Conversely, the first variation formula shows that every
point of $U_t\setminus U$ lies on an outward normal segment.  Hence the normal
exponential map parametrizes $U_t\setminus U$ without overlap.

Its Jacobian in the principal directions is
$(\cos s+\kappa_1\sin s)(\cos s+\kappa_2\sin s)$.
Set $U_{\pi/2}=\bigcup_{0<t<\pi/2}U_t$.  Letting $t\uparrow\pi/2$ and using
monotone convergence gives
\[
 |U_{\pi/2}|=|U|+\frac\pi4|\partial U|
 +\frac12\int_{\partial U}H_{\Sph^3}\,\mathrm dA
 +\frac\pi4\int_{\partial U}\kappa_1\kappa_2\,\mathrm dA.
\]
The Gauss equation and the Gauss--Bonnet theorem give
$\int_{\partial U}(1+\kappa_1\kappa_2)\,\mathrm dA=4\pi$.
Thus
\[
 |U_{\pi/2}|=|U|+\frac12\int_{\partial U}H_{\Sph^3}\,\mathrm dA+\pi^2.
\]
Since $|U_{\pi/2}|\leq|\Sph^3|=2\pi^2$, the claim follows.
\end{proof}

\medskip
\noindent\textit{Why codimension four appears.}
A four-dimensional normal fiber has radial directions in $\Sph^3$, so
$\partial U$ is a surface.  The highest curvature term in the spherical
Steiner formula is $\kappa_1\kappa_2$, and the Gauss equation and
Gauss--Bonnet turn
$\int_{\partial U}(1+\kappa_1\kappa_2)\,\mathrm dA$ into the constant
$4\pi$.  In codimension $m\geq5$, the corresponding boundary has dimension
at least three, and no analogous topological identity isolates its
mean-curvature integral.  This is the obstruction to extending the present
argument beyond codimension four.

The nonsmooth version of \eqref{A} uses the first spherical curvature
measure.

\begin{lemma}\label{Spherical curvature approximation}
Let $U$ be a geodesically convex body contained in an open hemisphere of
$\Sph^3$.  The local spherical Steiner formula uniquely defines a
nonnegative finite Borel measure $M_1(U,\cdot)$ on $\partial U$, called the
first spherical curvature measure.  It has the following properties:
\begin{enumerate}[label=\textup{(\roman*)}]
\item on every $C^2$ part $W$ of $\partial U$,
\[
 M_1(U,W)=\int_W H_{\Sph^3}\,\mathrm dA;
\]
\item if $\mathsf d_{\mathrm H}(U_j,U)\to0$, with all bodies contained in a
fixed open hemisphere, then $M_1(U_j,\cdot)$ converges weakly to
$M_1(U,\cdot)$ and $|U_j|\to|U|$.
\end{enumerate}
\end{lemma}

\begin{proof}
We use the normalization in which the smooth density of $M_1$ is the sum of
the principal curvatures.  In the smooth case, it is fixed by the term
\[
 \int_0^t(\kappa_1+\kappa_2)\sin s\cos s\,\mathrm ds
\]
in the normal Jacobian.  A geodesically convex body in an open hemisphere is
a spherical convex body in the sense required by Kohlmann: gnomonic
projection sends it to a Euclidean convex body, which can be approximated by
adding a small ball and smoothing its support function
\cite[Sections~1.8 and 2.5]{Schneider}.  Kohlmann's local
Steiner formula \cite[Theorem~2.7]{Kohlmann} therefore gives existence,
uniqueness, locality, and nonnegativity.  Weak continuity under Hausdorff
convergence is part of the spherical support-measure theory; see, in the same
normalization, \cite[Theorem~1 and the discussion following it]{Glasauer}.
Finally, volume convergence follows
from the almost-everywhere convergence of the indicator functions of convex
bodies and dominated convergence.
\end{proof}

To obtain the nonsmooth estimate, apply a gnomonic projection from the fixed
hemisphere to $\R^3$; it maps geodesic segments to line segments.  Adding a
small Euclidean ball and smoothing the support function produces smooth
strictly convex bodies converging in $\mathsf d_{\mathrm H}$.  Applying the
inverse gnomonic map gives smooth geodesically convex bodies $U_j$ with
$\mathsf d_{\mathrm H}(U_j,U)\to0$ that remain in the same hemisphere.  Lemma
\ref{Spherical Lemma} and Lemma \ref{Spherical curvature approximation} then
give
\begin{equation}\label{B}
 |U|+\frac12M_1(U,\partial U)\leq\frac{|\Sph^3|}{2}.
\end{equation}

For a bounded convex set $B\subset\R^4$, let $\cH^3$ denote
three-dimensional Hausdorff measure and define
\[
 A_B(r)=\cH^3\bigl(\{\theta\in\Sph^3:r\theta\in B\}\bigr),
 \qquad F_B(r)=r^2A_B(r).
\]
We use the standard distributional conventions for BV functions
\cite[Chapter~3]{AmbrosioFuscoPallara}.

\begin{lemma}\label{Radial Lemma}
The function $F_B$ belongs to $\mathrm{BV}_{\mathrm{loc}}(0,\infty)$ and
\begin{equation*}
\mathrm dF_B(r)\leq2|\Sph^3|r\,\mathrm dr
\end{equation*}
in the sense of distributions.  If $0\notin\operatorname{int}B$, then
\begin{equation}\label{C}
 \mathrm dF_B(r)\leq|\Sph^3|r\,\mathrm dr.
\end{equation}
\end{lemma}

\begin{proof}
Replacing $B$ by its closure does not change the associated distribution, so
we may assume that $B$ is compact.  If $0\in B$, convexity implies that
$A_B$ is nonincreasing.  Hence
\[
 \mathrm d(r^2A_B)=2rA_B\,\mathrm dr+r^2\,\mathrm dA_B
 \leq2rA_B\,\mathrm dr\leq2|\Sph^3|r\,\mathrm dr.
\]
This proves the first estimate.  If, in addition, $0\in\partial B$, a
supporting hyperplane through the origin gives $A_B(r)\leq|\Sph^3|/2$, and
the same computation proves \eqref{C}.

Suppose now that $0\notin B$.  On almost every ray, the intersection
with $B$ is empty or a compact interval $[a(\theta),b(\theta)]$; on empty
rays set $a(\theta)=b(\theta)=\infty$.  Set
\[
 U_r=\{\theta:a(\theta)\leq r\},\qquad
 E_r=\{\theta:b(\theta)<r\}.
\]
Then, for almost every $r$, $F_B(r)=r^2\bigl(|U_r|-|E_r|\bigr)$.
Both families are increasing.  Since $r^2|E_r|$ is nondecreasing, it suffices
to prove
\begin{equation}\label{Entry Variation}
 \mathrm d\bigl(r^2|U_r|\bigr)\leq|\Sph^3|r\,\mathrm dr.
\end{equation}

Assume first that $B=\Pi$ is a full-dimensional polytope.  Strong separation
gives $e\in\Sph^3$ and $\varepsilon>0$ such that
$\langle e,x\rangle\geq\varepsilon$ on $\Pi$.  Moreover,
$U_r=\Sph^3\cap\operatorname{cone}(\Pi\cap\overline{\B_r^4})$.
Thus every nonempty $U_r$ is geodesically convex and lies in a fixed open
hemisphere.  In a spherical cell for which the entry point lies in a visible
facet, write
\[
 \langle\nu,x\rangle=h,\qquad \Pi\subset\{\langle\nu,x\rangle\leq h\}.
\]
Visibility means $h<0$, and $a(\theta)=h/\langle\nu,\theta\rangle$.
At a moving boundary point $a(\theta)=r$, denote
\[
 \alpha=\langle\theta,\nu\rangle=\frac hr<0,\qquad
 s=\sqrt{1-\alpha^2},\qquad
 m=\frac{\nu-\alpha\theta}{s},\qquad \kappa=-\frac\alpha s.
\]
Here $m$ is the outward unit conormal of $U_r$ and
$|\nabla^{\Sph^3}a|=r/\kappa$.
On every compact interval of radii, $a$ is a finite, piecewise smooth locally
Lipschitz function on the relevant cells; the cell boundaries and critical
points are $\cH^3$-null.  The coarea formula \cite[\S3.2]{Federer} therefore
shows that
$r\mapsto|U_r|$ is locally absolutely continuous and, for almost every $r$,
\begin{equation}\label{Entry Coarea}
 \frac{\mathrm d}{\mathrm dr}|U_r|=\frac1r\int_{\Gamma_r}\kappa\,\mathrm dA,
\end{equation}
where $\Gamma_r$ is the moving part of $\partial U_r$ and $\mathrm dA$ is its
area measure.  Equivalently, the
push-forward of $\cH^3$ by $a$ has no singular part on such an interval; this
also follows by applying coarea on each cell.

Both principal curvatures of every smooth moving face, with respect to $m$,
are equal to $\kappa$.  Locality in the spherical Steiner formula
\cite[Theorem~2.7]{Kohlmann} implies
that on the relative interior of each moving $C^2$ stratum,
$\mathrm dM_1=2\kappa\,\mathrm dA$, while its restriction to each stationary
totally geodesic stratum vanishes.  The remaining part of $M_1$, supported on
the lower-dimensional strata, is nonnegative.  Therefore
\begin{equation}\label{above bound}
 2\int_{\Gamma_r}\kappa\,\mathrm dA\leq M_1(U_r,\partial U_r).
\end{equation}
Combining \eqref{Entry Coarea}, \eqref{above bound}, and
\eqref{B}, we obtain
\[
 \frac{\mathrm d}{\mathrm dr}\bigl(r^2|U_r|\bigr)
 =r\left(2|U_r|+\int_{\Gamma_r}\kappa\,\mathrm dA\right)
 \leq r\left(2|U_r|+\frac12M_1(U_r,\partial U_r)\right)
 \leq|\Sph^3|r.
\]
This proves \eqref{Entry Variation} for polytopes.

For a general compact convex set $B$ disjoint from the origin, choose
full-dimensional polytopes $\Pi_j$ with
$\mathsf d_{\mathrm H}(\Pi_j,B)\to0$, all contained in the same separating
open half-space; such approximation is standard in the Hausdorff topology
\cite[Section~1.8]{Schneider}.  Then
$\mathbf1_{\Pi_j}\to\mathbf1_B$ almost everywhere in $\R^4$.  Polar coordinates
and dominated convergence give
$F_{\Pi_j}\longrightarrow F_B$ in $L^1_{\mathrm{loc}}(0,\infty)$.
Passing to the limit in
$\mathrm dF_{\Pi_j}\leq|\Sph^3|r\,\mathrm dr$ proves
\eqref{C}.  Since
$|\Sph^3|r\,\mathrm dr-\mathrm dF_B$ is a positive distribution, it is a
Radon measure.  This also gives
$F_B\in \mathrm{BV}_{\mathrm{loc}}(0,\infty)$.
\end{proof}

For a convex set $B\subset\B_R^4$, set
\[
 J_B(s)=s\,\cH^3\big(\{\theta\in\Sph^3:\sqrt{s}\,\theta\in B\}\big),
 \qquad 0<s<R^2.
\]
Lemma \ref{Radial Lemma} gives
\begin{equation}\label{formula 2}
 \begin{cases}
  \mathrm dJ_B\leq|\Sph^3|\,\mathrm ds,
    &\text{always},\\
  \mathrm dJ_B\leq\dfrac{|\Sph^3|}{2}\,\mathrm ds,
    &\text{if }0\notin\operatorname{int}B.
 \end{cases}
\end{equation}
We use the left-continuous representative on $(0,R^2]$ and set $J_B(0)=0$.
The estimate $0\leq J_B(s)\leq|\Sph^3|s$ shows that this is the correct trace
at zero.  With this convention, $\mathrm dJ_B\leq\Lambda\,\mathrm ds$ is
equivalent to
\[
 J_B(s_2)-J_B(s_1)\leq\Lambda(s_2-s_1),
 \qquad\text{for every }0\leq s_1<s_2\leq R^2.
\]

Fix the following Borel representative.
Write $r_+=\max\{r,0\}$.  If $q$ is a locally integrable Borel function on
$(0,R^2)$, set
\begin{equation}\label{left average representative}
 \mathcal L^-q(s)=\limsup_{k\to\infty}
 k\int_{(s-1/k)_+}^{s}q(r)\,\mathrm dr
 \quad (s>0),
 \qquad \mathcal L^-q(0)=0.
\end{equation}
Whenever $q$ represents a locally BV function, $\mathcal L^-q$ is its
left-continuous representative.  This construction preserves joint Borel
measurability in any additional parameters because every average in
\eqref{left average representative} is Borel.

For a continuous concave function, every nonempty superlevel set is convex.
The next lemma applies layer cake to the radial sections of these sets.

\begin{lemma}\label{Weighted radial lemma}
Let $A\subset\B_R^4$ be relatively closed and convex, and let
$g:A\to[0,a]$ be continuous and concave.  For almost every $s\in(0,R^2)$ set
\begin{equation*}
 J_0(s)=s\int_{\Sph^3}g(\sqrt{s}\,\theta)^n
 \mathbf1_A(\sqrt{s}\,\theta)\,\mathrm d\theta.
\end{equation*}
Here $g$ is extended by zero outside $A$, and $\mathrm d\theta$ is the area
measure on $\Sph^3$.  Then $J_0$ has a
left-continuous $\mathrm{BV}$ representative $J$ on $(0,R^2]$, with $J(0)=0$,
and
\begin{equation}\label{Weighted bound}
 J(s_2)-J(s_1)\leq|\Sph^3|a^n(s_2-s_1).
\end{equation}
Set $\gamma=g(0)$ when $0\in A$ and $\gamma=0$ otherwise.  Then
\begin{equation}\label{Refined Weighted bound}
 J(s_2)-J(s_1)
 \leq\frac{|\Sph^3|}{2}(a^n+\gamma^n)(s_2-s_1).
\end{equation}
If $R>0$, $a>0$, and for some $h_j\downarrow0$,
\begin{equation}\label{endpoint slope equality}
 \frac{J(R^2)-J(R^2-h_j)}{h_j}\longrightarrow|\Sph^3|a^n,
\end{equation}
then $A=\B_R^4$ and $g\equiv a$.
\end{lemma}

\begin{proof}
For $0<t<a$, let $B_t=\{y\in A:g(y)>t\}$.
These sets are convex.  Let $J_{B_t}^0$ denote the defining radial expression
before choosing a representative, and set
$\widehat J_t=\mathcal L^-J_{B_t}^0$.  The map
$(t,s)\mapsto J_{B_t}^0(s)$ is Borel by its spherical integral, so
\eqref{left average representative} makes $(t,s)\mapsto\widehat J_t(s)$
jointly Borel.  Define
\begin{equation}\label{def of L}
 L(s)=n\int_0^a t^{n-1}\widehat J_t(s)\,\mathrm dt.
\end{equation}
The layer cake formula and Fubini theorem give $L=J_0$ almost everywhere.
Moreover, $0\leq\widehat J_t(s)\leq|\Sph^3|R^2$, so dominated convergence
shows that $L$ is left-continuous, and $L(0)=0$.

Applying \eqref{formula 2} under the integral in
\eqref{def of L} yields, in the sense of distributions,
$\mathrm dL\leq|\Sph^3|a^n\,\mathrm ds$.
Hence $L-|\Sph^3|a^n s$ agrees almost everywhere with a nonincreasing
function.  Since $L$ is bounded and left-continuous, it follows that
$L\in \mathrm{BV}(0,R^2)$; we take $J=L$.  Since $L=J_0$ almost everywhere
and $L$ is left-continuous, it also agrees with $\mathcal L^-J_0$.  The finite
difference form gives
\eqref{Weighted bound}.
If $t<\gamma$, then $0\in B_t$.  If $t\geq\gamma$, then
$0\notin\operatorname{int}B_t$.  Splitting the integral at $\gamma$
and using both bounds in \eqref{formula 2}, we obtain
\[
 \mathrm dJ\leq n|\Sph^3|
 \left(\int_0^\gamma t^{n-1}\,\mathrm dt
 +\frac12\int_\gamma^a t^{n-1}\,\mathrm dt\right)\mathrm ds
 =\frac{|\Sph^3|}{2}(a^n+\gamma^n)\,\mathrm ds.
\]
This proves \eqref{Refined Weighted bound}.

Assume \eqref{endpoint slope equality}.  The refined estimate gives
$|\Sph^3|a^n\leq\frac{|\Sph^3|}{2}(a^n+\gamma^n)$.
Hence $\gamma=a$, so $0\in A$, $g(0)=a$, and $0\in B_t$ for every $t<a$.
Put $\alpha_t(r)=\cH^3\bigl(\{\theta\in\Sph^3:r\theta\in B_t\}\bigr)$.
Since $0\in B_t$ and $B_t$ is convex, $\alpha_t$ is nonincreasing.  In terms of
left limits,
$\widehat J_t(R^2)=R^2\alpha_t(R-)$,
and therefore, for $0<h<R^2$,
\begin{equation}\label{endpoint slope bound}
 \frac{\widehat J_t(R^2)-\widehat J_t(R^2-h)}h
 \leq \alpha_t(R-)\leq|\Sph^3|.
\end{equation}
Using \eqref{def of L}, integrating
\eqref{endpoint slope bound}, and passing to the sequence in
\eqref{endpoint slope equality}, we obtain
\[
 |\Sph^3|a^n
 \leq n\int_0^a t^{n-1}\alpha_t(R-)\,\mathrm dt
 \leq|\Sph^3|a^n.
\]
Thus $\alpha_t(R-)=|\Sph^3|$ for almost every $t\in(0,a)$.  By monotonicity,
$\alpha_t(r)=|\Sph^3|$ for every $r<R$ and almost every such $t$.  Hence $B_t$
is dense on every sphere $r\Sph^3$, and its relative closure in $\B_R^4$ is
the whole ball.  Since $B_t\subset A$ and $A$ is relatively closed,
$A=\B_R^4$.
Choose $t_j\uparrow a$ among the preceding values.  Continuity of $g$ and the
density of $B_{t_j}$ give $g\geq t_j$ on $\B_R^4$.  Letting $j\to\infty$
proves $g\equiv a$.
\end{proof}

\begin{proof}[Proof of Theorem \ref{Fiberwise estimate theorem} in codimension four]
Set
\[
 b(x)=\big\|\det Q_x\,\mathbf1_{\cC_x}\big\|_{L^\infty(N_x\Sigma)},
 \qquad
 a(x)^n=\frac{b(x)}{\vartheta(x)}.
\]
The functions $a$ and $b$ are measurable.  Indeed, in a local trivialization
of the normal bundle, for every $t>0$ the set
$\{b>t\}$ is characterized by the positivity of the Lebesgue measure of
\[
 \{y\in\cC_x:\det Q_x(y)>t\},
\]
and measurability follows from Fubini's theorem.  Since $\mathcal E$ and
$\Sigma$ are bounded and $Q_x$ is affine in the normal variable, $a$ is
bounded.  The function $w_x$ is also bounded.

For $x\in\Omega$ with $a(x)>0$, define
\[
 g_x(z)=\left(
 \frac{\det Q_x(y_x^0+G_{NN}^{-1/2}z)}{\vartheta(x)}
 \right)^{1/n}
 \qquad\text{on }\widetilde{\cC}_x.
\]
The function $g_x$ is continuous and concave, because $Q_x$ is affine in the
normal variable and $A\mapsto(\det A)^{1/n}$ is concave on the positive
semidefinite cone \cite[Chapter~II]{Bhatia}.  Since $a(x)>0$, the convex fiber
$\widetilde{\cC}_x$ has positive four-dimensional measure and hence nonempty
interior.  Continuity then identifies the essential and pointwise suprema, so
$0\leq g_x\leq a(x)$.

For $x\in\Omega$ and almost every $0<s<R(x)^2$, set
\[
 J_x^0(s)=s\int_{\Sph^3}
 \frac{\det Q_x(y_x^0+G_{NN}^{-1/2}\sqrt{s}\,\theta)}{\vartheta(x)}
 \mathbf1_{\widetilde{\cC}_x}(\sqrt{s}\,\theta)\,\mathrm d\theta.
\]
Let $J_x=\mathcal L^-J_x^0$ on $(0,R(x)^2]$ and $J_x(0)=0$; outside
$\Omega$ put $R(x)=0$ and $J_x\equiv0$.  The integrand defining $J_x^0$ is
Borel in $(x,s,\theta)$, so \eqref{left average representative} makes
$(x,s)\mapsto J_x(s)$ jointly Borel.  If $a(x)=0$, then
$\det Q_x\mathbf1_{\cC_x}=0$ almost everywhere and $J_x\equiv0$.
If $a(x)>0$, apply Lemma \ref{Weighted radial lemma} to $g_x$.  In either case,
\begin{equation}\label{fiberwise finite difference}
 J_x(s_2)-J_x(s_1)
 \leq|\Sph^3|a(x)^n(s_2-s_1),
 \qquad0\leq s_1<s_2\leq R(x)^2.
\end{equation}

We pass from ambient volume to the radial data of the individual fibers.
Restricting the area formula to $\sqrt q\,\mathcal E$ produces an identity
with one radial integral on each fiber.  Differentiating in $q$ removes that
integral and allows the finite-difference estimate above to be applied
fiberwise.

For $0<q<1$, the area formula \eqref{contact area formula} with
$\varphi=\mathbf1_{\sqrt q\mathcal E}$, followed by
\eqref{Schur identity}, polar coordinates, and $s=|z|^2$, gives
\[
 c_{\mathcal E}\omega_{n+4}q^{\frac{n+4}{2}}
 =\frac12\int_\Sigma w_x
 \int_0^{(q-\tau_x)_+}J_x(s)\,\mathrm ds\,\mathrm d\mu(x).
\]
Since $J_x(0)=0$, \eqref{fiberwise finite difference} gives
\[
 0\leq J_x(s)\leq|\Sph^3|a(x)^n s.
\]
For fixed $x$, the function
\[
 q\longmapsto\int_0^{(q-\tau_x)_+}J_x(s)\,\mathrm ds
\]
is absolutely continuous and has derivative $J_x((q-\tau_x)_+)$ for almost
every $q$.  Moreover, $R(x)^2\leq1$, so this derivative is bounded by
$|\Sph^3|a(x)^n$.  Since $w_xa(x)^n$ is integrable, Fubini's theorem and
dominated convergence show that the integral over $\Sigma$ is absolutely
continuous and may be differentiated for almost every $q\in(0,1)$.  Thus
\begin{equation}\label{differentiated volume identity}
 (n+4)c_{\mathcal E}\omega_{n+4}q^{\frac{n+2}{2}}
 =\int_\Sigma w_xJ_x((q-\tau_x)_+)\,\mathrm d\mu(x).
\end{equation}
For $q_1<q_2$ in the full-measure set where this identity holds,
\eqref{fiberwise finite difference} gives
\[
 (n+4)c_{\mathcal E}\omega_{n+4}
 \left(q_2^{\frac{n+2}{2}}-q_1^{\frac{n+2}{2}}\right)
 \leq|\Sph^3|(q_2-q_1)\int_\Sigma w_xa(x)^n\,\mathrm d\mu.
\]
Dividing by $q_2-q_1$, letting $q_1,q_2\uparrow1$, and using
\begin{equation}\label{Ball Volume}
 (n+4)(n+2)\omega_{n+4}=2|\Sph^3|\omega_n
\end{equation}
we obtain
\begin{equation}\label{fiberwise determinant bound}
 c_{\mathcal E}\omega_n
 \leq\int_\Sigma w_x\frac{b(x)}{\vartheta(x)}
 \,\mathrm d\mu(x).
\end{equation}
By \eqref{conditional density identity} and
$|\mathcal E|=c_{\mathcal E}\omega_{n+4}$, this is \eqref{fiberwise estimate}.

We prove the rigidity assertion.  Assume equality.  Then
\begin{equation}\label{equality mass identity}
 \int_\Sigma w_xa(x)^n\,\mathrm d\mu=c_{\mathcal E}\omega_n.
\end{equation}
Letting $q\uparrow1$ through the full-measure set in
\eqref{differentiated volume identity}, and using
\eqref{fiberwise finite difference} as a dominating bound, gives
\[
 \int_\Sigma w_xJ_x(R(x)^2)\,\mathrm d\mu
 =(n+4)c_{\mathcal E}\omega_{n+4}.
\]
Choose $h_j\downarrow0$ such that $1-h_j$ belongs to the full-measure set in
\eqref{differentiated volume identity}.  Equality in the averaged estimate
forces the finite-difference bound to be saturated at the outer radius on
almost every nonzero fiber.  To measure the defect, put
\[
 \delta_h(x)=w_x\left[
 |\Sph^3|a(x)^n
 -\frac{J_x(R(x)^2)-J_x((R(x)^2-h)_+)}h
 \right].
\]
By \eqref{fiberwise finite difference}, $\delta_h\geq0$.  Using
\eqref{equality mass identity}, \eqref{differentiated volume identity},
and \eqref{Ball Volume},
\[
 \int_\Sigma\delta_h\,\mathrm d\mu
 =|\Sph^3|c_{\mathcal E}\omega_n
 -(n+4)c_{\mathcal E}\omega_{n+4}
 \frac{1-(1-h)^{\frac{n+2}{2}}}h\longrightarrow0.
\]
After passing to a subsequence, $\delta_{h_j}(x)\to0$ for almost every $x$.
If $a(x)>0$, equivalently $p_\infty(x)>0$, this implies $R(x)>0$, and for all
sufficiently large $j$,
\[
 \frac{J_x(R(x)^2)-J_x(R(x)^2-h_j)}{h_j}
 \longrightarrow|\Sph^3|a(x)^n.
\]
The last assertion of Lemma \ref{Weighted radial lemma}, applied to
$A=\widetilde{\cC}_x$ and $g=g_x$, proves
\eqref{fiberwise rigidity}.\qedhere
\end{proof}

\begin{proof}[Proof of Theorem \ref{Fiberwise estimate theorem} in lower codimension]
Let $1\leq m<4$, put $k=4-m$, and write
$\mathcal E=\{\xi:\langle G\xi,\xi\rangle<1\}$ and
$c_{\mathcal E}=(\det G)^{-1/2}$.  Set
\[
 b(x)=\big\|\det Q_x\,\mathbf1_{\cC_x}\big\|_{L^\infty(N_x\Sigma)}.
\]
For $L>0$, embed the immersion as $X_L=(X,0)$ in
$\R^{n+m}\times\R^k$ and set
\[
 \mathcal E_L
 =\{(\xi,z):\langle G\xi,\xi\rangle+L^{-2}|z|^2<1\}.
\]
The boundary condition is unchanged because
$h_{\mathcal E_L}(\eta,0)=h_{\mathcal E}(\eta)$.  The new normal directions
are flat, so the multiplicity is unchanged and
\[
 Q_{L,x}(y,z)=Q_x(y).
\]
The global supporting inequality is independent of $z$, so the extended
contact fiber is
\[
 \cC_{L,x}=\bigl\{(y,z):y\in\cC_x,
 |z|^2<L^2\bigl(1-\langle G(v(x)+y),v(x)+y\rangle\bigr)\bigr\}.
\]
Its projection onto $N_x\Sigma$ is $\cC_x$, and every $y\in\cC_x$ has a
nonempty open $z$-slice.  Thus the essential supremum of the determinant is
still $b(x)$.  On the other hand,
\[
 w_{L,x}=L^k w_x,
 \qquad c_{\mathcal E_L}=L^k c_{\mathcal E}.
\]
Applying the codimension-four determinant bound
\eqref{fiberwise determinant bound} to $(X_L,\mathcal E_L)$ and cancelling
$L^k$ gives
\[
 c_{\mathcal E}\omega_n
 \leq\int_\Sigma w_x\frac{b(x)}{\vartheta(x)}\,\mathrm d\mu(x).
\]
Since $|\mathcal E|=c_{\mathcal E}\omega_{n+m}$, the conditional density
identity \eqref{conditional density identity} proves
\eqref{fiberwise estimate}.  The flat example following the theorem proves
sharpness.
\end{proof}

\section{Applications}\label{section applications}

\subsection{The anisotropic Michael--Simon--Sobolev inequality}

We first prove Corollary \ref{Main Corollary} in codimension four with
$\mathcal E\subset\R^{n+4}$.  Lower codimensions are handled at the end by
flat extension.  Assume that $\Sigma$ is connected and put
\[
 M=\int_\Sigma f^{n/(n-1)}\,\mathrm d\mu.
\]
The left-hand side of \eqref{Main inequality} is positive.  If
$\partial\Sigma\neq\varnothing$, its boundary term is positive.  If
$\partial\Sigma=\varnothing$ and the left-hand side vanishes, then
$\nabla^\Sigma f+fH=0$.  Its tangential and normal parts vanish separately,
so $H=0$, and Stokes' theorem gives
\[
 0=\int_\Sigma\Delta_\Sigma|X|^2\,\mathrm d\mu
  =\int_\Sigma\bigl(2n+2\langle X,H\rangle\bigr)\,\mathrm d\mu
  =2n|\Sigma|,
\]
a contradiction.  By homogeneity, we may assume
\begin{equation}\label{Normalization}
 \int_\Sigma h_{\mathcal E}(\nabla^\Sigma f+fH)\,\mathrm d\mu
 +\int_{\partial\Sigma}f h_{\mathcal E}(\eta)\,\mathrm d\sigma=nM.
\end{equation}

The boundary datum in the auxiliary equation below is chosen to reproduce
the boundary term in \eqref{Main inequality}, as in the ABP proof of the Wulff
inequality \cite{CabreRosOtonSerra}.  Consider the Neumann problem
\begin{align}
 \diver_\Sigma(f\nabla^\Sigma u)
 &=nf^{\frac n{n-1}}
   -h_{\mathcal E}(\nabla^\Sigma f+fH)
 &&\text{in }\Sigma,\label{Neumann equation}\\
 \langle\nabla^\Sigma u,\eta\rangle&=h_{\mathcal E}(\eta)
 &&\text{on }\partial\Sigma.\label{Neumann boundary}
\end{align}
By \eqref{Normalization},
\[
 \int_\Sigma\left(nf^{\frac n{n-1}}
 -h_{\mathcal E}(\nabla^\Sigma f+fH)\right)\,\mathrm d\mu
 =\int_{\partial\Sigma}f h_{\mathcal E}(\eta)\,\mathrm d\sigma,
\]
which is the compatibility condition.  Hence
\eqref{Neumann equation}--\eqref{Neumann boundary} has a solution unique up to
an additive constant.  Since $f>0$, division by $f$ gives a uniformly
elliptic equation with smooth coefficients.  The function
$h_{\mathcal E}(z)=\sqrt{\langle G^{-1}z,z\rangle}$ is Lipschitz, so the
interior right-hand side belongs to $C^{0,\alpha}$ for every
$\alpha\in(0,1)$.  The boundary operator is the outward normal derivative and
is therefore uniformly oblique.  Moreover, $|\eta|=1$ and $h_{\mathcal E}$ is
smooth away from the origin, so $x\mapsto h_{\mathcal E}(\eta(x))$ is smooth.
Global Schauder theory for the oblique derivative problem
\cite[Section~6.7]{GilbargTrudinger}
therefore gives $u\in C^{2,\alpha}(\Sigma)$ for every $\alpha\in(0,1)$.  We
apply the construction of Section~\ref{section inverse transport} to this
$u$.

\begin{lemma}\label{Determinant lemma}
Set $Z=\nabla^\Sigma\log f+H\in\R^{n+4}$.  For every $(x,y)\in\cC$,
\begin{equation}\label{Trace identity}
 \tr Q_x(y)=nf(x)^{\frac{1}{n-1}}
 -\big(h_{\mathcal E}(Z(x))+\langle Z(x),\Phi(x,y)\rangle\big),
\end{equation}
and
\begin{equation}\label{Determinant inequality}
 0\leq\bigl(\det Q_x(y)\bigr)^{\frac{1}{n}}
 \leq\frac{\tr Q_x(y)}n\leq f(x)^{\frac{1}{n-1}}.
\end{equation}
\end{lemma}

\begin{proof}
Dividing \eqref{Neumann equation} by $f$ gives
\[
 \Delta_\Sigma u=nf^{\frac{1}{n-1}}-h_{\mathcal E}(Z)
 -\langle\nabla^\Sigma\log f,v\rangle.
\]
Since $v$ is tangent and $y$ is normal,
$\langle\nabla^\Sigma\log f,v\rangle+\langle H,y\rangle
=\langle Z,\Phi(x,y)\rangle$.  Subtracting $\langle H,y\rangle$ therefore
gives \eqref{Trace identity}.  Moreover,
$\mathcal E$ is centrally symmetric and $\Phi(x,y)\in\mathcal E$, so
$h_{\mathcal E}(Z)\geq\langle Z,-\Phi(x,y)\rangle$.  Combining
\eqref{QPositive}, \eqref{Trace identity}, and the arithmetic--geometric mean
inequality \cite{Bhatia} gives \eqref{Determinant inequality}.
\end{proof}

\begin{proof}[Proof of Corollary \ref{Main Corollary}]
\medskip

\noindent\textbf{Step 1: The sharp inequality.}
Set
\[
 b(x)=\big\|\det Q_x\,\mathbf1_{\cC_x}\big\|_{L^\infty(N_x\Sigma)}.
\]
Theorem \ref{Fiberwise estimate theorem}, in the equivalent form
\eqref{fiberwise determinant bound}, and Lemma \ref{Determinant lemma} give
\begin{align}
 c_{\mathcal E}\omega_n
 &\leq\int_\Sigma w_x\frac{b(x)}{\vartheta(x)}\,\mathrm d\mu(x)
 \notag\\
 &\leq\int_\Sigma
 w_x\frac{f^{\frac n{n-1}}}{\vartheta(x)}\,\mathrm d\mu(x)
 \notag\\
 &=c_{\mathcal E}\omega_n
 \int_\Sigma\frac{f^{\frac n{n-1}}}{\vartheta(x)V_x}\,\mathrm d\mu(x)
 \notag\\
 &\leq\frac{c_{\mathcal E}\omega_n}{V_*(\mathcal E)}M.
 \label{Mass Bound}
\end{align}
where we used \eqref{Schur Jacobian}, $\vartheta\geq1$, and
$V_x\geq V_*(\mathcal E)$.  Cancelling $c_{\mathcal E}\omega_n$ gives
$M\geq V_*(\mathcal E)$.  Undoing the normalization proves
\eqref{Main inequality} for connected submanifolds of codimension four.

\medskip

\noindent\textbf{Step 2: The equality case.}
Assume equality in \eqref{Main inequality} under the normalization
\eqref{Normalization}.  Then $M=V_*(\mathcal E)$ and every inequality in
\eqref{Mass Bound} is an equality.  Since $f>0$,
\begin{equation}\label{Equality Masses}
 \vartheta(x)=1,
 \qquad V_x=V_*(\mathcal E),
 \qquad b(x)=f(x)^{\frac n{n-1}}
\end{equation}
for almost every $x$.  In particular, $p_\infty>0$ almost everywhere.
Equality also holds in \eqref{fiberwise estimate}.  By the rigidity part of
Theorem \ref{Fiberwise estimate theorem}, for almost every $x$,
\[
 \widetilde{\cC}_x=\B^4_{R(x)},
 \qquad
 \det Q_x(y)=f(x)^{\frac n{n-1}}
 \quad\text{for every }y\in\cC_x.
\]
In particular, $R(x)>0$ almost everywhere.  By \eqref{Schur identity}, the
identity $\widetilde{\cC}_x=\B^4_{R(x)}$ says that, for these $x$, $\cC_x$ is
the entire normal slice
\[
 \{y\in N_x\Sigma:v(x)+y\in\mathcal E\}.
\]
Hence equality holds in \eqref{Determinant inequality} for every point of this
open slice, and therefore
\begin{align}
 Q_x(y)&=f(x)^{\frac{1}{n-1}}I_{T_x\Sigma}
 &&\text{whenever }v(x)+y\in\mathcal E,\label{QEquality}\\
 h_{\mathcal E}(Z(x))+\langle Z(x),v(x)+y\rangle&=0
 &&\text{whenever }v(x)+y\in\mathcal E.\label{Trace Equality}
\end{align}
Choose any such $y$ and put $\xi=v(x)+y\in\mathcal E$.  If $Z(x)\neq0$,
then $-\xi$ is an interior point of $\mathcal E$, and therefore
$h_{\mathcal E}(Z(x))>\langle Z(x),-\xi\rangle$, contrary to
\eqref{Trace Equality}.  Hence $Z(x)=0$.  Its tangential and normal parts give
$\nabla^\Sigma f(x)=0$ and $H(x)=0$.  Comparing the coefficients of $y$ in
\eqref{QEquality}, which holds on an open subset of the normal fiber, gives
\begin{equation}\label{Identities}
 \mathrm{II}(x)=0,
 \qquad D^2_\Sigma u(x)=f(x)^{\frac{1}{n-1}}I_{T_x\Sigma}.
\end{equation}
These identities hold on a full-measure, hence dense, subset of $\Sigma$.
All terms are continuous, so they hold everywhere.

It follows that $f$ is constant.  Since the differential of the Gauss map is
determined by $\mathrm{II}$, the identity $\mathrm{II}=0$ and the connectedness
of $\Sigma$ imply that the tangent plane is constant.  Thus the image is
contained in an affine $n$-plane parallel to some linear plane $P$, and the
restriction of $X$ to the interior is a local diffeomorphism into that affine
plane.  By \eqref{Equality Masses} and continuity,
$|\operatorname{proj}_P\mathcal E|=V_*(\mathcal E)$.
The boundary cannot be empty, because otherwise $X(\Sigma)$ would be both
open and compact in the affine plane.

Let $c>0$ be such that $f\equiv c^{n-1}$ and put
$K=\operatorname{proj}_P\mathcal E$.  Identifying tangent vectors with their
images under $dX$, \eqref{Identities} shows that the ambient derivative of
$\nabla^\Sigma u-cX$ vanishes.  Hence
$\nabla^\Sigma u=c(X-x_0)$ for some point $x_0$ in the affine plane.  Since
$R(x)>0$ almost everywhere, continuity gives $\tau_x\leq1$ on $\Sigma$.
Here $v=c(X-x_0)\in P$, the Schur complement $S_x$ is independent of $x$,
and $K=\{z\in P:\langle S_xz,z\rangle<1\}$.  Thus
\begin{equation}\label{image in Wulff body}
 X(\Sigma)\subset x_0+c^{-1}\overline K.
\end{equation}
For $z\in P$, the definition of the support function gives
\[
 h_K(z)=\sup_{\xi\in\mathcal E}
 \langle z,\operatorname{proj}_P\xi\rangle=h_{\mathcal E}(z).
\]
On the boundary, \eqref{Neumann boundary} therefore becomes
$\langle c(X-x_0),\eta\rangle=h_K(\eta)$.  Together with
\eqref{image in Wulff body}, this says that
$c(X-x_0)\in\partial K$ at every boundary point.  No interior point can map to
$x_0+c^{-1}\partial K$, since $X$ is a local diffeomorphism into the affine
plane and its image lies in the closed convex body.  Therefore
$X:\operatorname{int}\Sigma\to x_0+c^{-1}K$ is a local diffeomorphism.  It is
proper: the preimage of a compact subset of $x_0+c^{-1}K$ is closed in the
compact space $\Sigma$ and, because the boundary maps to
$x_0+c^{-1}\partial K$, is contained in $\operatorname{int}\Sigma$.

A proper local diffeomorphism is a covering map \cite{Lee}.  Its image is
therefore nonempty, open, and closed in the connected set $x_0+c^{-1}K$, hence
is the whole set.  The interior of a connected manifold with boundary is
connected \cite{Lee}, while the ellipsoid $K$ is simply connected.
Consequently, the covering has one sheet, and $X$ is a diffeomorphism between
the interiors.

Because $X$ is an immersion and maps boundary to boundary,
\[
dX(T_x\partial\Sigma)
=
T_{X(x)}\partial\bigl(x_0+c^{-1}\overline K\bigr)
\quad\text{for every }x\in\partial\Sigma.
\]
Since $dX$ has rank $n$, the image of an inward conormal to
$\partial\Sigma$ has a nonzero component transverse to the target
boundary.  The inclusion $X(\Sigma)\subset x_0+c^{-1}\overline K$
forces this component to point inward.  Hence the inverse function theorem
for manifolds with boundary shows that $X$ is a diffeomorphism from a
collar neighborhood of each boundary point onto a relative neighborhood
in $x_0+c^{-1}\overline K$.  If two distinct boundary points had the same
image, choose disjoint such collar neighborhoods.  Their images would
contain the same smaller inward
half-neighborhood, giving interior points with two preimages and
contradicting the one-sheeted covering.  Consequently, $X$ is an
embedding onto $x_0+c^{-1}\overline K$.

Conversely, suppose that $f$ is constant and $X(\Sigma)$ is a translate of
$r\overline K$ for a minimizing plane $P$.  After translating the image,
assume $X(\Sigma)=r\overline K$.  If $x\in\partial(rK)$ and $\eta$ is its
outward unit normal, then
$h_K(\eta)=r^{-1}\langle x,\eta\rangle$.  The divergence theorem gives
\[
 \int_{\partial(rK)}h_K(\eta)\,\mathrm d\sigma
 =\frac1r\int_{\partial(rK)}\langle x,\eta\rangle\,\mathrm d\sigma
 =nr^{n-1}|K|.
\]
Together with $h_{\mathcal E}|_P=h_K$ and $|rK|=r^n|K|$, this gives equality
in \eqref{Main inequality}.

\medskip

\noindent\textbf{Step 3: Disconnected submanifolds and lower codimensions.}
Suppose now that $\Sigma=\bigsqcup_j\Sigma_j$ is disconnected, and put
$M_j=\int_{\Sigma_j}f^{\frac n{n-1}}\,\mathrm d\mu$.  Summing the connected
inequality over the components gives
\[
 \int_\Sigma h_{\mathcal E}(\nabla^\Sigma f+fH)\,\mathrm d\mu
 +\int_{\partial\Sigma}f h_{\mathcal E}(\eta)\,\mathrm d\sigma
 \geq nV_*(\mathcal E)^{\frac{1}{n}}\sum_jM_j^{\frac{n-1}{n}}
 \geq nV_*(\mathcal E)^{\frac{1}{n}}
 \left(\sum_jM_j\right)^{\frac{n-1}{n}}.
\]
The last inequality is strict if at least two $M_j$ are positive.  Since
$f>0$, equality forces $\Sigma$ to be connected.

Finally, let $m<4$.
The fiber estimate in these dimensions was already obtained at the end of
Section~\ref{section fiberwise estimate}.  Here we choose the added semiaxes
so that the minimum projection volume, and hence the equality case, is also
preserved.  After an orthogonal change of coordinates, write the
semiaxes of $\mathcal E$ as
$0<a_1\leq\cdots\leq a_{n+m}$.  Embed $\R^{n+m}$ linearly into
$\R^{n+4}$ and extend $\mathcal E$ to an ellipsoid
$\widetilde{\mathcal E}\subset\R^{n+4}$ by adjoining $4-m$ semiaxes all
equal to some $L>a_{n+m}$.  For vectors $z$ in the original ambient space,
$h_{\widetilde{\mathcal E}}(z,0)=h_{\mathcal E}(z)$.  Moreover, the minimum
volume of an $n$-dimensional projection of an ellipsoid with ordered semiaxes
$b_1\leq\cdots\leq b_{n+4}$ is $\omega_n b_1\cdots b_n$.  Indeed, take
$A=\operatorname{diag}(b_1,\ldots,b_{n+4})$ and write the ellipsoid as $A\B^{n+4}$.
If the columns of a $(n+4)\times n$ matrix $U$ form an orthonormal basis of the
projection plane, the projected volume is $ \omega_n\det(U^{\mathsf T}A^2U)^{1/2}.$
The min--max principle for the eigenvalues of $U^{\mathsf T}A^2U$
\cite{Bhatia} gives
$\det(U^{\mathsf T}A^2U)\geq b_1^2\cdots b_n^2$, with equality on the first $n$
principal axes.  Hence
$V_*(\widetilde{\mathcal E})=V_*(\mathcal E)$.

The extra normal directions are parallel and flat, so the induced metric,
$\mathrm{II}$, $H$, and the boundary measure are unchanged.  The codimension
four result applied to $\widetilde{\mathcal E}$ proves the inequality for
every $1\leq m\leq4$.

In the equality case, the image is an $n$-dimensional convex body with
nonempty relative interior and is already contained in the original ambient
space.  Its affine span, and hence its tangent plane $P$, therefore lies in
that ambient space.  Moreover,
$\operatorname{proj}_P\widetilde{\mathcal E}
 =\operatorname{proj}_P\mathcal E$, so the preceding argument gives the
equality case in Corollary \ref{Main Corollary}.
\end{proof}

Since $H\perp T\Sigma$, for $\mathcal E=\B^{n+m}$ one has
$h_{\mathcal E}(\nabla^\Sigma f+fH)
=\sqrt{|\nabla^\Sigma f|^2+f^2|H|^2}$.  Hence Corollary
\ref{Main Corollary} contains the following isotropic form.

\begin{corollary}
Let $n\geq2$ and $1\leq m\leq4$.  Under the assumptions on $X$ and $\Sigma$
in Corollary \ref{Main Corollary}, every positive $f\in C^\infty(\Sigma)$
satisfies
\[
 \int_\Sigma \sqrt{|\nabla^\Sigma f|^2+f^2|H|^2}\,\mathrm d\mu
 +\int_{\partial\Sigma}f\,\mathrm d\sigma
 \geq n\omega_n^{1/n}
 \left(\int_\Sigma f^{\frac n{n-1}}\,\mathrm d\mu\right)^{\frac{n-1}{n}}.
\]
Equality holds if and only if $f$ is constant and $X$ is an embedding whose
image is a flat round $n$-ball.
\end{corollary}

\begin{corollary}[Sharp ellipsoidal Wulff inequality]\label{Wulff corollary}
Let $n\geq2$ and $1\leq m\leq4$.  Let
$X:\Sigma\to\R^{n+m}$ be a compact minimal immersion with smooth boundary,
and let $\mathcal E\subset\R^{n+m}$ be a centered ellipsoid.  Then
$\partial\Sigma\ne\varnothing$ and
\[
 \int_{\partial\Sigma}h_{\mathcal E}(\eta)\,\mathrm d\sigma
 \geq nV_*(\mathcal E)^{1/n}|\Sigma|^{\frac{n-1}{n}}.
\]
Equality holds if and only if, for an $n$-plane $P$ realizing
$V_*(\mathcal E)$, $X$ is an embedding onto
$x_0+r\,\overline{\operatorname{proj}_P\mathcal E}$ for some
$x_0\in\R^{n+m}$ and $r>0$.
\end{corollary}

\subsection{Convex symmetrization on minimal submanifolds}
\label{section symmetrization}

A compact minimal immersion into Euclidean space cannot be boundaryless, as
$\Delta_\Sigma|X|^2=2n$.  Thus the minimal submanifolds in this section have
nonempty boundary, and $W_0^{1,p}(\Sigma)$ is understood in the usual trace
sense.

Fix an $n$-plane $P$ realizing $V_*(\mathcal E)$ and set
\[
 K=\operatorname{proj}_P\mathcal E\subset P,
 \qquad |K|=V_*(\mathcal E).
\]
Throughout this subsection, $K$ denotes this fixed projected ellipsoid.
We denote the Euclidean gradient and Lebesgue measure on $P$ by $\nabla^P$
and $\mathrm dz$.  This is the convex symmetrization associated with $K$
\cite{AlvinoFeroneLionsTrombetti,EspositoTrombetti}; see also
\cite{BrothersZiemer} for the classical rearrangement framework.  For a
nonnegative measurable function $u$ on $\Sigma$,
let
\[
 \mu_u(t)=|\{u>t\}|,
 \qquad
 r_u(t)=\left(\frac{\mu_u(t)}{|K|}\right)^{1/n},
\]
and define its $K$-symmetrization on $P$ by
\[
 u^K(z)=\int_0^\infty\mathbf1_{r_u(t)K}(z)\,\mathrm dt.
\]
Thus $u$ and $u^K$ are equimeasurable and the superlevel sets of $u^K$ are
homothetic copies of $K$.  Since $r_u(t)\leq r_u(0)$, the function $u^K$ is
supported in the compact set $\overline{r_u(0)K}\subset P$.

\begin{theorem}[Sharp anisotropic P\'olya--Szeg\H{o} principle]
\label{Polya Szego theorem}
Let $n\geq2$ and $1\leq m\leq4$, and let
$X:\Sigma\to\R^{n+m}$ be a compact minimal immersion with smooth boundary.
Let $\mathcal E\subset\R^{n+m}$ be a centered ellipsoid, choose $P$ and $K$
as above, and let $1<p<\infty$.  If $u\in W_0^{1,p}(\Sigma)$ is
nonnegative, then $u^K\in W^{1,p}(P)$ and
\begin{equation}\label{anisotropic Polya Szego}
 \int_P h_K(\nabla^P u^K)^p\,\mathrm dz
 \leq
 \int_\Sigma h_{\mathcal E}(\nabla^\Sigma u)^p\,\mathrm d\mu.
\end{equation}
The constant $1$ in \eqref{anisotropic Polya Szego} is optimal.
\end{theorem}

For $1<p<n$, put $p^*=np/(n-p)$ and
\[
 A_{n,p}=\pi^{-1/2}n^{-1/p}
 \left(\frac{p-1}{n-p}\right)^{1-1/p}
 \left(
  \frac{\Gamma(1+n/2)\Gamma(n)}
       {\Gamma(n/p)\Gamma(1+n-n/p)}
 \right)^{1/n}.
\]
Here $\Gamma$ denotes Euler's gamma function.  This is the sharp Euclidean
Sobolev constant of Aubin and Talenti
\cite{Aubin,Talenti}.

\begin{corollary}[Sharp anisotropic $L^p$-Sobolev inequality]
\label{anisotropic Lp Sobolev corollary}
Under the assumptions of Theorem \ref{Polya Szego theorem}, every
$u\in W_0^{1,p}(\Sigma)$, $1<p<n$, satisfies
\begin{equation}\label{anisotropic Lp Sobolev}
 \left(\int_\Sigma |u|^{p^*}\,\mathrm d\mu\right)^{1/p^*}
 \leq A_{n,p}
 \left(\frac{\omega_n}{V_*(\mathcal E)}\right)^{1/n}
 \left(\int_\Sigma
  h_{\mathcal E}(\nabla^\Sigma u)^p\,\mathrm d\mu\right)^{1/p}.
\end{equation}
The constant is optimal uniformly over the stated class.  In particular,
$\mathcal E=\B^{n+m}$ gives the sharp inequality
\[
 \|u\|_{L^{p^*}(\Sigma)}
 \leq A_{n,p}\|\nabla^\Sigma u\|_{L^p(\Sigma)}
\]
for minimal submanifolds of codimension at most four.
\end{corollary}

\begin{proof}[Proof of Theorem \ref{Polya Szego theorem}]
We first take $u\geq0$ smooth with zero boundary values.  For a regular value
$t>0$, set
\[
 \Omega_t=\{u>t\},
 \qquad \nu_t=-\frac{\nabla^\Sigma u}{|\nabla^\Sigma u|}.
\]
Because $u=0$ on $\partial\Sigma$ and $t>0$, one has
$\overline{\Omega_t}\Subset\operatorname{int}\Sigma$.
For every nonempty regular superlevel set,
$\overline{\Omega_t}$ is a compact smooth manifold with boundary $\{u=t\}$,
and
$X|_{\overline{\Omega_t}}:\overline{\Omega_t}\longrightarrow\R^{n+m}$
is a compact minimal immersion.  Since $\mathcal E$ is centrally symmetric,
$h_{\mathcal E}(-z)=h_{\mathcal E}(z)$.  Corollary
\ref{Wulff corollary}, applied to $\overline{\Omega_t}$, therefore gives
\begin{equation}\label{level set Wulff}
 I(t):=\int_{\{u=t\}}h_{\mathcal E}(\nu_t)\,\mathrm d\sigma
 \geq n|K|^{1/n}\mu_u(t)^{\frac{n-1}{n}}.
\end{equation}
For almost every $t$, the coarea formula \cite[Section~3.2]{Federer} yields
\[
 -\mu_u'(t)=
 \int_{\{u=t\}}\frac{1}{|\nabla^\Sigma u|}\,\mathrm d\sigma.
\]
H\"older's inequality on the level hypersurface gives
\[
 I(t)^p
 \leq
 \left(\int_{\{u=t\}}
  h_{\mathcal E}(\nu_t)^p|\nabla^\Sigma u|^{p-1}
  \,\mathrm d\sigma\right)\bigl(-\mu_u'(t)\bigr)^{p-1}.
\]
A second application of coarea, followed by \eqref{level set Wulff}, gives
\begin{align}
 \int_\Sigma h_{\mathcal E}(\nabla^\Sigma u)^p\,\mathrm d\mu
 &=\int_0^\infty\int_{\{u=t\}}
   h_{\mathcal E}(\nu_t)^p|\nabla^\Sigma u|^{p-1}
 \,\mathrm d\sigma\,\mathrm dt \notag\\
 &\geq \int_0^\infty
   \frac{I(t)^p}{(-\mu_u'(t))^{p-1}}\,\mathrm dt \notag\\
 &\geq \int_0^\infty
   \frac{\bigl(n|K|^{1/n}\mu_u(t)^{(n-1)/n}\bigr)^p}
        {(-\mu_u'(t))^{p-1}}\,\mathrm dt.       \label{one dimensional energy}
\end{align}
The quotient in \eqref{one dimensional energy} is understood as zero when
$\mu_u(t)=0$ and as $+\infty$ when $\mu_u(t)>0$ but $-\mu_u'(t)=0$.

For the $K$-symmetrization, the standard one-dimensional representation
\cite{AlvinoFeroneLionsTrombetti,EspositoTrombetti} gives
\begin{equation}\label{symmetrized energy formula}
 \int_P h_K(\nabla^P u^K)^p\,\mathrm dz
 =\int_0^\infty
   \frac{\bigl(n|K|^{1/n}\mu_u(t)^{(n-1)/n}\bigr)^p}
        {(-\mu_u'(t))^{p-1}}\,\mathrm dt.
\end{equation}
This formula uses the generalized inverse of $\mu_u$ and follows by monotone
approximation.  Its geometric terms can be seen as follows.  Let
$\rho_K(z)=\inf\{r>0:z\in rK\}$ be the Minkowski functional of $K$.
Since $K$ is a smooth strictly convex ellipsoid,
\begin{equation}\label{gauge support identity}
 h_K(\nabla^P\rho_K)=1
 \qquad\text{on }P\setminus\{0\}.
\end{equation}
The function $u^K$ is nonincreasing in $\rho_K$ and has distribution function
$\mu_u$.  The anisotropic perimeter of the level set $r_u(t)K$ is
$n|K|^{1/n}\mu_u(t)^{(n-1)/n}$.
Central symmetry of $K$ and \eqref{gauge support identity} show that
$h_K(\nabla^P u^K)$ depends only on $\rho_K$ and is therefore constant on
almost every regular level set.  Thus the
level-set H\"older inequality is an equality for $u^K$.  Combining
\eqref{one dimensional energy} and \eqref{symmetrized energy formula} proves
\eqref{anisotropic Polya Szego} for smooth $u$.

For general $u\in W_0^{1,p}(\Sigma)$, choose nonnegative
$u_j\in C_c^\infty(\operatorname{int}\Sigma)$ with $u_j\to u$ in $W^{1,p}$.
The $L^p$ contraction property gives $u_j^K\to u^K$ in $L^p(P)$.
The smooth inequality gives a uniform bound for
$h_K(\nabla^P u_j^K)$ in $L^p(P)$.  Since $h_K$ is equivalent to the Euclidean
norm and $1<p<\infty$, weak compactness in $W^{1,p}(P)$ and the strong
$L^p$ convergence imply that $u^K\in W^{1,p}(P)$.  Lower semicontinuity,
the Lipschitz continuity of $h_{\mathcal E}$, and the convergence of $u_j$ in
$W^{1,p}(\Sigma)$ give
\eqref{anisotropic Polya Szego}.

For optimality, take a flat domain in $P$ containing the support of a smooth
function that is decreasing in $\rho_K$.  On $P$ one has
$h_{\mathcal E}=h_K$, and the function agrees with its $K$-symmetrization.
Thus equality is attained in \eqref{anisotropic Polya Szego}, so the factor
one is optimal.
\end{proof}

\begin{proof}[Proof of Corollary \ref{anisotropic Lp Sobolev corollary}]
Apply the preceding theorem to $|u|$.  Since $\mathcal E$ is centrally
symmetric,
$h_{\mathcal E}(\nabla^\Sigma|u|)=h_{\mathcal E}(\nabla^\Sigma u)$ almost
everywhere.  Identify $P$ with
$\R^n$ and write $K=A\B^n$ for an invertible linear map $A$.  Then
\[
 h_K(\xi)=|A^{\mathsf T}\xi|,
 \qquad |\det A|=\frac{|K|}{\omega_n}.
\]
The change of variables $z=Ay$ in the sharp Euclidean inequality of Aubin and
Talenti gives, for every $v\in W^{1,p}(P)$,
\begin{equation}\label{anisotropic Sobolev on plane}
 \|v\|_{L^{p^*}(P)}
 \leq A_{n,p}\left(\frac{\omega_n}{|K|}\right)^{1/n}
 \left(\int_P h_K(\nabla^P v)^p\,\mathrm dz\right)^{1/p}.
\end{equation}
Equimeasurability, \eqref{anisotropic Sobolev on plane}, and Theorem
\ref{Polya Szego theorem} prove \eqref{anisotropic Lp Sobolev}.  The constant
is optimal by taking cut-off Aubin--Talenti extremals on expanding flat
domains in a minimizing plane $P$.
\end{proof}

\end{document}